\documentclass[a4paper,11pt]{article}
\usepackage{amsfonts}
\usepackage[T1]{fontenc}
\usepackage{microtype}
\usepackage{makeidx}
\usepackage{setspace}
\usepackage{graphicx}
\usepackage[all]{xy}
\usepackage{amsmath}
\usepackage{amssymb}
\usepackage{setspace}
\usepackage{booktabs}
\usepackage{braket}
\usepackage{array}
\usepackage{tabularx}
\usepackage{multirow}
\usepackage{indentfirst}
\usepackage{cite}
\usepackage{stmaryrd}
\usepackage{faktor}
\usepackage{titling}
\usepackage{tikz}
\usepackage{dsfont}
\usetikzlibrary{matrix,arrows,decorations.pathmorphing}

\usepackage{amsthm}
\theoremstyle{plain}
\newtheorem{thm}{Theorem}[section]
\newtheorem{dfn}[thm]{Definition}
\newtheorem{lem}[thm]{Lemma}
\newtheorem{prop}[thm]{Proposition}
\newtheorem{cor}[thm]{Corollary}
\newtheorem{ex}[thm]{Example}

\theoremstyle{remark}
\newtheorem{oss}[thm]{Remark}

\DeclareMathOperator{\id}{\mathrm{Id}}

\DeclareMathOperator{\End}{\mathrm{End}}
\DeclareMathOperator{\imm}{\mathrm{Im}}

\DeclareMathOperator{\sm}{\mathrm{sm}}
\DeclareMathOperator{\Or}{\mathcal{OR}}
\DeclareMathOperator{\N}{\mathbb{N}}
\DeclareMathOperator{\modue}{(\mathrm{mod}~2)}

\begin{document}

\begin{center}

{\Large \bf }

\vspace{0.5cm}

{\Large \bf Special idempotents and projections}

 \vspace{0.5cm}

 Paolo Sentinelli \\

$\,$ \\

{\em paolosentinelli@gmail.com } \\

\end{center}

\begin{abstract}
We define, for any special matching of a finite graded poset, an idempotent, regressive and order preserving function. We consider the monoid generated by such functions. We call \emph{special idempotent} any idempotent element of this monoid.
They are interval retracts. Some of them realize a kind of parabolic map and are called \emph{special projections}. We prove that, in Eulerian posets, the image of a special projection, and its complement, are graded induced subposets. In a finite Coxeter group, all projections on right and left parabolic quotients are special projections,
and some projections on double quotients too. We extend our results to special partial matchings.
\end{abstract}

\section{Introduction}

The notion of special matching of a partially ordered set has been introduced by F. Brenti in \cite{brenti-cohomology};
he proved that, in any symmetric group, the Kazhdan-Lusztig polynomial $P_{e,v}$ depends only on the isomorphism type
of the Bruhat interval $[e,v]$. Here $e$ denotes the identity of the group. The proof lies on the fact that these polynomials can be computed using only
special matchings of the Bruhat intervals. This result has been extended to all Coxeter groups by
Brenti, Caselli and Marietti in \cite{bcm}. In this note, for any special matching $M$ of a finite poset $K$, we define an idempotent function
$P^M : K \rightarrow K$, and, by composition, we consider the monoid generated by the set of such idempotents. It turns out that this is a submonoid
of the monoid of regressive order preserving functions $\Or(K)$ (Proposition \ref{morfismo}). We call \emph{special idempotents}
the idempotents of this submonoid. Any special idempotent gives a Galois connection, since this is true for any idempotent of $\Or(K)$ (Proposition \ref{proposizione galois});
moreover, extending a definition in \cite{forcey sottile}, we prove that special idempotents are \emph{interval retracts} (Proposition \ref{prop interv}) and they realize a partition of the poset $K$ into intervals (Corollary \ref{corollario partizione}). In the context of Coxeter groups we recover the partition into left, right or double cosets of parabolic subgroups.
In fact, in a finite Coxeter group $W$ with the Bruhat order,  all projections on quotients $W^J$, $^JW$ and double quotients ${^IW^J}$ are special
idempotents, obtained by composing the idempotents corresponding to multiplication matchings (see Section \ref{ultima sezione}).

The notion of \emph{projection} (Definition \ref{def proiezione}) is introduced in order to obtain a simple formula,
known for parabolic quotients of Coxeter groups, of the M\"obius function of the image of any idempotent regressive order preserving function which satisfies a maximum condition, when the poset is Eulerian (Corollary \ref{cor moebius}). We call \emph{special projections} the special idempotents which are projections.  Our main results are that, in Eulerian posets, the
image  of a special projection is a graded poset (Theorem \ref{immagine graduata}) and that
its complement is a graded poset (Theorem \ref{complemento graduato}), as induced subposets.
These theorems generalize, in the finite cases, well known results in Coxeter groups and a theorem of \cite{sentinelli comp}.
In fact, these last ones can be deduced by our theorems, once is proved the existence of a parabolic map, as Billey, Fan and Losonczy introduced in  \cite{billey},
and a coset extension of it (see Proposition \ref{prop parabolic map}),
or, following our terminology, once is proved that the projections on quotients $W^J$ and $^JW$ are special projections (Corollary \ref{teorema proiezioni}).

These results, together with Proposition \ref{moebius proiezioni}, allow us to deduce the M\"obius function of the posets $W^J$ and $^JW$ once is known the M\"obius function of $W$ with the Bruhat order. For double quotients ${^IW^J}$ the situation is different in general. It is known that these posets could be not graded.
For graded double quotients, Examples \ref{esempio non proiezione graduato} and \ref{esempi proiezioni Sn} show that the canonical projections on them could be not special projections;
the same is true if we consider the special idempotent $\hat{P}^J: \mathrm{Invol}(W) \rightarrow {^JW^J} \cap \mathrm{Invol}(W)$ defined in Section \ref{sezione involuzioni}, where $\mathrm{Invol}(W)$
is the set of involutions of the group $W$ (see Example \ref{esempi proiezioni involuzioni Sn}).

We devote the last section to extend our results to special partial matchings, which have been introduced by A. Hultman in \cite{Hultman1}
in order to prove a combinatorial invariance result in the spirit of Brenti's work cited above.

\section{Notation and preliminaries}

In this section we establish some notation and we collect some basic definitions and
results from the theory of finite posets
which will be useful in the sequel. We let $\mathbb{Z}$ be the set of integer numbers and $\mathbb{N}$ the set of non-negative
integers; for any $n\in \mathbb{N}$ we use the notation $[n]:=\{1,2,\ldots,n\}$; in
particular $[0]=\varnothing$. With $\biguplus$ we denote the disjoint union and with $|X|$ the cardinality of a set $X$. A function $\id_X : X \rightarrow X$ is defined by $\id_X(x)=x$, for all $x\in X$. For any function $f:X\rightarrow Y$ and any $y\in Y$ we define a subset of $X$ by
$f_y:=\{x \in X:f(x)=y\}$ and the image of $f$ by $\imm(f):=\{f(x):x\in X\}$.

Given a partially ordered set (poset) $K$, any pair $(x,y)\in K \times K$ satisfying
$x\leqslant y$ defines an \emph{interval}  $[x,y]:=\{z\in K:x\leqslant z \leqslant y\}$; when $|[x,y]|=2$ we say that $y$ \emph{covers} $x$ and we write $x \vartriangleleft y$.
In this article we consider a finite poset $K$ to be \emph{graded} if it has minimum, maximum and a rank function $\rho : K \rightarrow \N$.
We set $\rho(x,y):=\rho(y)-\rho(x)$, for all $x,y\in K$. We denote by $\hat{0}$ and $\hat{1}$ the minimum and the maximum of $K$, respectively.
We let $\mathrm{atom}(K):=\{x \in K: \hat{0} \vartriangleleft x \}$ be the set of \emph{atoms} of $K$.

The M\"obius function of a finite poset $K$ is the function $\mu_K: K\times K \rightarrow \mathbb{Z}$ defined by

$$\mu_K(x,y)=\left\{
               \begin{array}{ll}
                 1, & \hbox{if $x=y$;} \\
                 -\sum_{x \leqslant z < y}\mu_K(x,z), & \hbox{if $x<y$;} \\
                 0, & \hbox{otherwise,}
               \end{array}
             \right.$$
for all $(x,y)\in K\times K$. A graded poset $K$ with rank function $\rho$ is said to be \emph{Eulerian} if $\mu_K(x,y)=(-1)^{\rho(x,y)}$,
for all $x,y\in K$ such that $x\leqslant y$.

Given two posets $H$ and $K$, a \emph{Galois connection between $H$ and $K$} is a pair $(f,g)$ of order preserving functions $f: H \rightarrow K$
and $g: K \rightarrow H$ such that $x \leqslant g(y)$ if and only if $f(x) \leqslant y$, for all $x\in H$, $y\in K$.
Given a Galois connection $(f,g)$ between $H$ and $K$, the M\"obius functions of the two posets satisfy the following equality (see e.g. \cite{aguiar} and references therein):
\begin{equation} \label{rota}
\sum \limits_{y \in f_v}\mu_H(x,y)=\sum \limits_{u \in g_x}\mu_K(u,v),\end{equation}
for all $x\in H$, $v\in K$.

\section{The monoid of regressive poset endomorfisms}

Let $K$ be a finite poset. The category of finite posets considered here is the one whose morphisms are the order preserving functions.
A function $f\in \End(K)$ is said to be \emph{regressive} if $f(x) \leqslant x$, for all $x\in K$.
The set $\Or(K):=\{f\in \End(K): \mbox{$f$ is regressive}\}$
is a finite monoid, with composition of functions as operation and $\id_K$ as identity. The monoid of regressive order preserving functions of a finite poset has been considered in the literature; it is an example of $J$-trivial monoid.
See \cite{denton1} and references therein. From here to the end we write composition of functions by juxtaposition or using the symbol $\circ$.
The proof of the following lemma is straightforward.
\begin{lem} \label{prop punti fissi}
  Let $f\in \Or(K)$, $f=f_1\cdots f_k$ for some $f_1,\ldots,f_k \in \Or(K)$, and $v\in K$. Then $f(v)=v$ if and only if $f_i(v)=v$ for all $i \in [k]$.
\end{lem}

\begin{cor}
  Let $f\in \Or(K)$ be idempotent, $f=f_1\cdots f_k$ for some idempotent functions $f_1,\ldots,f_k \in \Or(K)$. Then
$$ \imm(f) = \bigcap \limits_{i=1}^k \imm(f_i).$$
\end{cor}

The following two results are known. We give proofs, in our setting, for sake of completeness (see, e.g., \cite[Section 3]{denton1}).

\begin{prop} \label{prop assorbimento}
  Let $f\in \Or(K)$, $f=f_1\cdots f_k$  for some $f_1,\ldots,f_k \in \Or(K)$. Then the following are equivalent:
\begin{enumerate}
  \item \label{a} $f$ is idempotent;
  \item \label{b} $f_if=f$, for all $i \in [k]$;
  \item \label{c} $ff_i=f$, for all $i \in [k]$.
\end{enumerate}
\end{prop}
\begin{proof}
  The implication \eqref{a} $\Rightarrow$ \eqref{b} follows directly by Lemma \ref{prop punti fissi} and the implication
 \eqref{c} $\Rightarrow$ \eqref{a} is obvious.
Assume $f_if=f$, for all $i \in [k]$. Then $f$ is idempotent. By hypothesis $f(v)\leqslant v$, $f_i(v) \leqslant v$, for all $v\in K$, $i\in [k]$, and these inequalities imply $f(v)=(ff_if)(v) \leqslant (ff_i)(v) \leqslant f(v)$, i.e. $(ff_i)(v)=f(v)$, for all $v\in K$.
\end{proof}

Let $\mathrm{E}(M)$ be the set of idempotents of a monoid $M$.
Such a set is partially ordered by letting $P \leqslant Q$ if and only if $PQ=QP=Q$, for all $P,Q\in E(M)$.

\begin{cor}
  The poset of idempotents $E(\Or(K))$ is a lattice.
\end{cor}
\begin{proof}
  The poset $E(\Or(K))$ has minimum $\id_K$. Let $P,Q\in E(\Or(K))$ be not comparable idempotents, and $M:=\langle P ,Q \rangle \subseteq \Or(K)$ the submonoid generated by $P$ and $Q$. The finiteness of $K$ and the regressivity of the functions $P$ and $Q$ imply the existence of an element $T\in \langle P ,Q \rangle$ such that $PT=QT=T$; hence, by Proposition \ref{prop assorbimento}, $T$
is idempotent and $P,Q<T$.
If $T_1,T_2 \in E(M)\setminus \{P,Q\}$, then, again by Proposition \ref{prop assorbimento}, $P,Q \leqslant T_1$ and $P,Q \leqslant T_2$; this implies
$T_1=T_1T_2=T_2$.  We have proved that $E(M)=\{\id_K,P,Q,T\}$. If $U\in E(\Or(K))$ satisfies $P,Q<U$, then
$T \leqslant U$; so we can define $P \vee Q := T$. Since $E(\Or(K))$ has minimum, it is a lattice.
\end{proof}

Let $P\in \Or(K)$ be idempotent. If $K$ has minimum $\hat{0}$ and maximum $\hat{1}$, then the minimum of $\imm(P)$, as induced subposet of $K$, is $P(\hat{0})=\hat{0}$
and its maximum is $P(\hat{1})$. Moreover we observe that $\min P_x = \{x\}$, for all $x \in \imm(P)$.

\begin{prop} \label{proposizione galois}
  Let $P\in \Or(K)$ be idempotent. Then the poset morphisms $\id_{\imm(P)}: \imm(P) \rightarrow K$ and $P:K \rightarrow \imm(P)$ give a Galois connection
between $\imm(P)$ and $K$.
\end{prop}
\begin{proof}
Let $x\in \imm(P)$ and $y\in K$.  Then $\id_{\imm(P)}(x) \leqslant y$ implies $x \leqslant y$ and then $x=P(x) \leqslant P(y)$, since $P$ is order preserving.
Let $x \leqslant P(y)$; then $x \leqslant P(y) \leqslant y$ by regressivity, i.e. $\id_{\imm(P)}(x) \leqslant y$.
\end{proof}

We use formula \eqref{rota} to deduce the following corollary.
\begin{cor} \label{cor moebius}
  Let $K$ be a finite poset and $P\in \Or(K)$ idempotent. Then
  $$ \mu_{\imm(P)}(x,y)=  \sum \limits_{z\in [x,y]\cap P_x } \mu_K(z,y),$$
   for all $x,y\in \imm(P)$, $x \leqslant y$.
\end{cor}
The formula in  Corollary \ref{cor moebius} suggests the next definition.
\begin{dfn} \label{def proiezione}
  We say that an idempotent $P\in \Or(K)$ is a \emph{projection} if $[x,y]\cap P_x$ is an interval of $K$,
  for all $x,y\in \imm(P)$, $x\leqslant y$.
  \end{dfn}
We see in Corollary \ref{teorema proiezioni} how projections in finite Coxeter groups are related to the existence of a parabolic map, in the meaning of \cite{billey}. For a projection of an Eulerian poset with rank function $\rho$, the formula of Corollary \ref{cor moebius} simplifies.

\begin{prop} \label{moebius proiezioni}
Let $K$ be Eulerian and $P\in \Or(K)$ be a projection. Then
  $$\mu_{\imm(P)}(x,y)=\left\{
                         \begin{array}{ll}
                           (-1)^{\rho(x,y)}, & \hbox{if $[x,y] \subseteq \imm(P)$;} \\
                           0, & \hbox{otherwise,}
                         \end{array}
                       \right.
$$ for all $x,y\in \imm(P)$, $x\leqslant y$.
\end{prop}
\begin{proof}
  If $P\in \Or(K)$ is a projection and $x,y\in \imm(P)$, $x < y$, then there exists $w\in K$ such that
$[x,y]\cap P_x =[x,w]$. So, by Corollary \ref{cor moebius},
\begin{eqnarray*}
  \mu_{\imm(P)}(x,y) &=& \sum \limits_{z\in [x,w] } (-1)^{\rho(z,y)} \\
  &=& (-1)^{\rho(w,y)}\sum \limits_{z\in [x,w] } (-1)^{\rho(z,w)} \\
&=&  \left\{
      \begin{array}{ll}
        (-1)^{\rho(x,y)}, & \hbox{if $w=x$;} \\
        0, & \hbox{if $w\neq x$.}
      \end{array}
    \right.
\end{eqnarray*}
Note that $[x,y] \subseteq \imm(P)$ implies $w=x$.
We claim that the converse holds too, and this proves the stated formula.
Assume $[x,y]\setminus \imm(P) \neq \varnothing$ and $w=x$.
Let $z\in \min ([x,y] \setminus \imm(P))$. Then $x < P(z) < z$. Since $K$ is Eulerian,
there exists $u\in [x,y]$ such that $x < u \vartriangleleft z$ and $u\neq P(z)$. Hence $u=P(u)< P(z)$, a contradiction.
\end{proof}

A projection of an Eulerian poset has the following property. We let, for any $A \subseteq \mathrm{atom}(K)$,
$$K_A:=\{x \in K: a  \leqslant x \Rightarrow a \in A, \, \, \forall \, a \in \mathrm{atom}(K)\}.$$
Notice that $K_A$ is an order ideal of $K$, and that $K_{\varnothing}=\{\hat{0}\}$. Moreover, if $P \in \Or(K)$ and $A:=P_{\hat{0}} \cap \mathrm{atom}(K)$  then
$P_{\hat{0}} \subseteq K_A$. In Eulerian posets the opposite inclusion holds, whenever $P$ is a projection.
\begin{prop}\label{proposizione atomi}
  Let $K$ be Eulerian, $P\in \Or(K)$ a projection and $A:=P_{\hat{0}} \cap \mathrm{atom}(K)$. Then $P_{\hat{0}}=K_A$.
\end{prop}
\begin{proof}
We prove that $K_A \subseteq P_{\hat{0}}$. Let $y\in \min(K_A \setminus P_{\hat{0}})$. Hence $\rho(y)>1$. 
        If $P(y)<y$ then $P(y)\in K_A$ and $P(y)=\hat{0}$ by the minimality of $y$. Then $y\in P_{\hat{0}}$, a contradiction. Let $P(y)=y$.
        Then there exist $w_1, w_2 \vartriangleleft y$, $w_1 \neq w_2$, since $K$ is Eulerian and $\rho(y)>1$. Moreover $w_1, w_2 \in K_A \cap P_{\hat{0}}$ and then
$\max([\hat{0},y]\cap P_{\hat{0}})\supseteq \{w_1,w_2\}$, a contradiction. Therefore $y \in P_{\hat{0}}$.
\end{proof}


We end this section with a generalization, in the finite case, of \cite[Lemma 5.1]{sentinelli comp}.
\begin{lem} \label{lemma complemento}
  Let $P\in \Or(K)$ be idempotent and $u,v\in K\setminus \imm(P)$ such that $u\leqslant v$. Then  $[u, v]\subseteq  K\setminus \imm(P)$ if and only if
$u \nless P(v)$.
\end{lem}
\begin{proof}
  If $u<P(v)$, then $P(v)\in [u,v]$ but $P(v)\not\in K\setminus \imm(P)$. Let $w \in [u, v]\cap \imm(P) \neq \varnothing$; then
  $w=P(w) \leqslant P(v)$. Therefore $u < w \leqslant P(v)$.
\end{proof}

\section{Special idempotents and special projections} \label{sezione special}

From here to the end $K$ is a finite graded poset with minimum $\hat{0}$, maximum $\hat{1}$
and rank function $\rho$. A function $M: K \rightarrow K$ is a \emph{matching}\footnote{In graph theory this corresponds to a \emph{perfect matching} of the Hasse diagram.} if
\begin{enumerate}
  \item $M \circ M = \id_K$;
  \item $M(x) \vartriangleleft x$ or $x \vartriangleleft M(x)$, for all $x\in K$.
\end{enumerate}
A matching $M: K \rightarrow K$ is a \emph{special matching} if  $M(x) \leqslant M(y)$ whenever $x \vartriangleleft y$ and $x \neq M(y)$, for all $x,y\in K$.
We refer to \cite{bcm}, \cite{bcm2}, \cite{bcm3}, \cite{bcm4}, \cite{caselli marietti} and \cite{caselli marietti2} for motivations and further deepening.
We let $\mathrm{m}(K)$ be the set of matchings of $K$ and $\sm(K)$ the one of special matchings.

\begin{oss}
 We observe that a matching $M\in \mathrm{m}(K)$ is special if and only if satisfies the \emph{lifting property}, namely:
  let $u,v\in K$ be such that $u \leqslant v$, $u \vartriangleleft M(u)$ and $M(v) \vartriangleleft v$. Then $M(u) \leqslant v$ and $u \leqslant M(v)$
  (see \cite[Proposition~2.2.7]{BB} for the corresponding property in Coxeter groups). One implication is the content of \cite[Lemma 4.2]{brenti-cohomology};
the other implication is easy to be deduced.
\end{oss}

We resume the previous remark in the following proposition.
\begin{prop}
  Let $M \in \mathrm{m}(K)$. Then $M \in \sm(K)$ if and only if $M$ satisfies the lifting property.
\end{prop}

Given a special matching $M\in \sm(K)$, we define an idempotent function $P^M : K \rightarrow K$ by setting

$$P^M(x)= \left\{
          \begin{array}{ll}
            x, & \hbox{if $x \vartriangleleft M(x)$;} \\
            M(x), & \hbox{if $M(x)\vartriangleleft x$,}
          \end{array}
        \right.
$$ for all $x\in K$. By definition, $P^M$ is regressive. We prove that
it is also order preserving, i.e. $P^M\in \Or(K)$.


\begin{prop} \label{morfismo}
  Let $M\in \sm(K)$ and $u,v\in K$. Then  $u \leqslant v$ implies $P^M(u)\leqslant P^M(v)$.
\end{prop}
\begin{proof}
If $v \vartriangleleft M(v)$ the result is straightforward. If $M(v) \vartriangleleft v $ the result follows easily
by using the lifting property.
%
\end{proof}

\begin{dfn} Let $M^K \subseteq \Or(K)$ be the submonoid generated by the set of {idempotents}\footnote{An $IG$-monoid, in the meaning of \cite{howie}.} $\{\id_K\}\cup \{P^{M}:M \in \sm(K)\}$. We call \emph{special idempotent} any idempotent of the monoid $M^K$.
\end{dfn}
As we see in Section \ref{ultima sezione}, in a Coxeter system $(W,S)$, if $I,J \subseteq S$, the idempotent function $Q^IP^J : W \rightarrow {^I}W^J$
which projects an element $w \in W$ into the representative of minimal length of the double coset $W_IwW_J$,
is a special idempotent.


\begin{thm} \label{teorema cartesiano}
  Let $K_1,\ldots,K_n$ be finite graded posets and $K:= K_1 \times \ldots \times K_n$.
Then $M\in \sm(K)$ if and only if $$M=\id_{K_1}\times \ldots \times \id_{K_{i-1}} \times N \times  \id_{K_{i+1}} \times \ldots \times \id_{K_n},$$ for some $i\in [n]$, $N \in \sm(K_i)$.
\end{thm}
\begin{proof} It is sufficient to prove the result for $n=2$.
The poset $K=K_1 \times K_2$ is finite and
 graded, with minimum $\hat{0}=(\hat{0}_1,\hat{0}_2)$.

We claim that $M(\hat{0}_1,\hat{0}_2) \in K_1 \times \{\hat{0}_2\}$ implies $M(K_1 \times \{\hat{0}_2\})\subseteq  K_1 \times \{\hat{0}_2\}$.
Let $x\in K_1\setminus \{\hat{0}_1\}$ and assume $M(y,\hat{0}_2)\in K_1 \times \{\hat{0}_2\}$ for all $y<x$;
define $(a,b):=M(x,\hat{0}_2)$. Then either $a\neq x$ and $b=\hat{0}_2$ or $a=x$ and $b \vartriangleright \hat{0}_2$.
Let $a=x$ and $y\in K_1$ be such that $y \vartriangleleft x$. Then $(y,b) \vartriangleleft (x,b)$ and $M(y,b) \neq (x,b)$; hence $M(y,b) \leqslant M(x,b)=(x,\hat{0}_2)$,
i.e. $M(y,b)=(z,\hat{0}_2)$, for some $z\in K_1$, $z<x$.
Therefore $(y,b)=M(z,\hat{0}_2)\in K_1 \times \{\hat{0}_2\}$, a contradiction. Hence $b=\hat{0}_2$ and the claim is proved.

Let $M_1: K_1 \rightarrow K_1$ be the function defined by $M(x,\hat{0}_2)=(M_1(x),\hat{0}_2)$, for all $x\in K_1$; then $M_1\in \sm(K_1)$.
It remains to prove that $M(x,y)=(M_1(x),y)$ for all $x,y$.
Let $(x,y)\in K$ be such that $y\neq \hat{0}_2$ and $M(u,v)=(M_1(u),v)$ for all $(u,v)<(x,y)$. There are two cases to be considered.
\begin{enumerate}
  \item $M(x,y) \vartriangleleft (x,y)$: in this case, if $M(x,y)=(u,v)$ then $(x,y)=M(M(x,y))=(M_1(u),v)$, i.e. $M_1(u)=x$ and $v=y$; hence $M(x,y)=(M_1(x),y)$.
  \item $(x,y) \vartriangleleft M(x,y)$: when $M_1(x) \vartriangleleft x$ we have that $M(M_1(x),y)=(M_1(M_1(x)),y)=(x,y)$. So let
    $x \vartriangleleft M_1(x)$ and consider $z\in K_2$ satisfying $z \vartriangleleft y$; then $(x,z) \vartriangleleft (x,y)$ and $M(x,z)\neq (x,y)$.
Therefore $M(x,z)=(M_1(x),z)\vartriangleleft M(x,y)$. We conclude by noting that
$M(x,y) \vartriangleright (x,y)$ and $M(x,y) \vartriangleright (M_1(x),z)$ imply $M(x,y)=(M_1(x),y)$.
\end{enumerate}
Since either $M(\hat{0}_1,\hat{0}_2) \in K_1 \times \{\hat{0}_2\}$ or $M(\hat{0}_1,\hat{0}_2) \in \{\hat{0}_1\} \times K_2$, we have proved that
either $M=M_1 \times \id_{K_2}$ or $M=\id_{K_1}\times M_2$, for some $M_1 \in \sm(K_1)$, $M_2\in \sm(K_2)$.
The reverse implication in our statement is straightforward.
\end{proof}

Under the hypothesis of Theorem \ref{teorema cartesiano} we have $\Or(K_1)\times \ldots \times \Or(K_n) \hookrightarrow \Or(K)$, as monoids.
When considering the submonoid $M^K$, we have an isomorphism.
\begin{cor}
   Let $K_1,\ldots,K_n$ be finite graded posets and $K:= K_1 \times \ldots \times K_n$.
Then, as monoids, $M^K \simeq M^{K_1}\times \ldots \times M^{K_n}$.
\end{cor}

\begin{ex}
Let $n\in \N$, $n>1$ and $n=p_1^{k_1}\cdots p_h^{k_h}$ be the prime factorization of $n$. Let $P_n:=\{z\in \N: z\mid n\}$ ordered by divisibility.
This poset is isomorphic to a Cartesian product of chains: $P_n \simeq c_{1+k_1} \times \ldots \times c_{1+k_h}$.
Let $i\in [h]$ be such that $k_i \equiv 1 \modue$, and $M_i : P_n \rightarrow P_n$ the function defined by
$$M_i(z)=\left\{
    \begin{array}{ll}
      zp_i, & \hbox{if $v_{p_i}(z)\equiv 0\modue$;} \\
      z/p_i, & \hbox{if $v_{p_i}(z)\equiv 1\modue$,}
    \end{array}
  \right.
$$ for all $z\in P_n$, where $v_p(z)$ is the $p$-adic valuation of $z$. Then  $M_i \in \sm(P_n)$.
Moreover, by Theorem \ref{teorema cartesiano},  we have  that $\sm(P_n)=\{M_i:k_i \equiv 1 \modue\}$.
Let $m:=|\{i\in [h]:k_i \equiv 1 \modue\}|$. We have proved that, as monoids, $M^{P_n}\simeq (\mathcal{P}([m]),\cup)$.
Notice that $P_n$ is a zircon (in the meaning of \cite{marietti-zirconi}) if and only if $k_i=1$ for all $i\in [h]$.
\end{ex}

The following definition appears in \cite{forcey sottile}. A surjective poset
morphism $f : H \rightarrow K$ is an {\emph{interval retract}}\footnote{In the original definition $H$ is considered to be a finite lattice.} if $f_x$ is an
interval, for all $x\in K$, and there exists a poset morphism $g : K \rightarrow H$ such that $f \circ g = \id_K$.
The next proposition asserts that special idempotents are interval retracts.
\begin{prop} \label{prop interv}
 Let $P \in \Or(K)$ be a special idempotent and $v\in \imm(P)$. Then $P_v$ is an interval of $K$.
\end{prop}
\begin{proof} Let $v\in \imm(P)$. A special idempotent is order preserving so we need only to prove that $P_v$ has minimum and maximum.
We have already observed that $\min P_v = \{v\}$.
We prove that $P_v$ has maximum.
If $P=\id_K$ then $P_v=\{v\}$ for all $v\in K$, and the maximum is $v$.
 Let $M: K \rightarrow K$ be a special matching. Hence $P_v=\{v,M(v)\}$ and  then
$$\max(P_v)=\left\{
                            \begin{array}{ll}
                              M(v), & \hbox{ if $v \vartriangleleft M(v)$;} \\
                              v, & \hbox{ if $M(v) \vartriangleleft v$.}
                            \end{array}
                          \right.$$

 Let $P \in M^K$ be such that $P=P^{M_1}\cdots P^{M_k}$, with $M_1,\ldots,M_k \in \sm(K)$. We proceed by induction on $k$,
 the case $k=1$ having been already discussed. We have that
 $v\in \imm(P')$, by Lemma \ref{prop punti fissi}; here we have defined $P':=P^{M_1}\cdots P^{M_{k-1}}$.
 By induction, $P'_v$ has maximum $w$.
 We claim that $$\max(P_v)=\left\{
                            \begin{array}{ll}
                              M_k(w), & \hbox{ if $w \vartriangleleft M_k(w)$;} \\
                              w, & \hbox{ if $M_k(w) \vartriangleleft w$.}
                            \end{array}
                          \right.$$
Notice that, if $x \in P_v\setminus P'_v$ then $x=M_k(u) \vartriangleright u$, for some $u \in P'_v$. We discuss the two possible cases.
  \begin{enumerate}
    \item $M_k(w) \vartriangleleft w$: $u \leqslant w$ and $u \vartriangleleft M_k(u)$ imply $M_k(u) \leqslant w$, by the lifting property.
    \item $w \vartriangleleft M_k(w)$: $u \leqslant w$ and $u \vartriangleleft M_k(u)$ imply $M_k(u) \leqslant M_k(w)$, by the lifting property.
  \end{enumerate}
\end{proof}
We denote by $v^P$ the maximum of $P_v$, for $v$ in the image of any special idempotent $P$. The following corollary states that any special idempotent of $\Or(K)$ gives a partition of $K$ into intervals.
 \begin{cor} \label{corollario partizione}  Let $P\in \Or(K)$ be a special idempotent. Then
  $$K=\biguplus \limits_{v\in \imm(P)} [v,v^P].$$
\end{cor}


%

We give now some results about projections in $\Or(K)$ which are special idempotents, whenever $K$ is Eulerian. We start with a definition.

\begin{dfn}
  We say that a projection $P\in \Or(K)$ is \emph{special} if $P\in M^K$.
\end{dfn}
By Proposition \ref{prop interv}, if $P$ is a special projection then $[x,y] \cap [x,x^P]$ is an interval of $K$, for all $x,y\in \imm(P)$, $x \leqslant y$.
It is straightforward to see that if $M$ is a special matching of $K$ then $P^M$ is a special projection.
Examples of special idempotents which are not projections are given in Section \ref{ultima sezione}, see Examples \ref{esempio non proiezione graduato}, \ref{esempi proiezioni Sn} and \ref{esempi proiezioni involuzioni Sn}.

The next two theorems are the main results of this article.
\begin{thm} \label{immagine graduata}
  Let $K$ be Eulerian and $P\in \Or(K)$ a projection. Then $\imm(P)$ is graded with rank function $\rho$.
\end{thm}
\begin{proof} We have already noted that $\imm(P)$ has maximum and minimum.
Let $u,v \in \imm(P)$ be such that $u\leqslant v$ and $\rho(u,v)>1$.
Let $C(u,v):=\{z\in K: u \leqslant z \vartriangleleft v\}$ be the set of coatoms of $[u,v]$. Then $|C(u,v)|>1$, since $K$ is Eulerian. Let us assume $z \not \in \imm(P)$, for all $z\in C(u,v)$.
Then $P(z)\neq P(y)$, being $P$ a projection, and $[P(y),y] \cap [P(z),z] = \varnothing$, for all $y,z\in C(u,v)$.
Let $y\in C(u,v)$ such that $P(y)$ is a maximal element of the set $\{P(z):z\in C(u,v)\}$.
 If $z\in [P(y),v]$ for some $z \in C(u,v)\setminus \{y\}$ then $P(y)<z$ and this implies $P(y)<P(z)$, a contradiction.
 Since the coatoms of $[P(y),v]$ are coatoms of $[u,v]$, we conclude that $[P(y),v]$ is not Eulerian, which is a contradiction. Then there exists  $z \in \imm(P) \cap C(u,v)$ and this proves the result.
%
%
%
\end{proof}

The previous result is well known for the idempotent functions $P^J: W \rightarrow W^J$ and $Q^J: W \rightarrow {^JW}$,
where $(W,S)$ is a Coxeter system and $J \subseteq S$ (see, e.g. \cite[Theorem 2.5.5 and Corollary 2.7.10]{BB} and the next section).
We end with a generalization, in the finite case, of \cite[Theorem 5.2]{sentinelli comp}.

\begin{thm} \label{complemento graduato}
  Let $K$ be Eulerian and $P\in \Or(K)$ a special projection. Then the poset $[K\setminus \imm(P)]\cup \{\hat{0}\}$ is graded with rank function $\rho$.
\end{thm}
\begin{proof} Notice that $P(\hat{1})=\hat{1}$ if and only if $P=\id_K$. Hence, if $P\neq \id_K$, the set $K\setminus \imm(P)$ has maximum $\hat{1}$.  Let $P=P^{M_1}\cdots P^{M_k}$ for some $M_1,\ldots,M_k$ special matchings of $K$.
  Let $u,v\in K\setminus \imm(P)$ such that $u\leqslant v$ and $\rho(u,v)>1$.
  We claim that there exists $z\not \in \imm(P)$ such that $u \leqslant z \vartriangleleft v$.
  If $u \nless P(v)$ then the result follows by Lemma \ref{lemma complemento}.
  Let $u < P(v)$. We proceed by induction on $\rho(u,v)$. Let $\rho(u,v)=2$; then $[u,v]=\{u,v,P(v),z\}$, for some $z\in K$.
  If $z\in \imm(P)$ then $P^{M_i}(z)=z$ for all $i\in [k]$; moreover there exits $i\in [k]$ such that $M_i(v) \vartriangleleft v$.
  Then $z \vartriangleleft M_i(z)$ and, by the lifting property, $z \leqslant M_i(v)=P(v)$, a contradiction.
  Let $\rho(u,v)>2$. We let $y:=\max\left([P(u),P(v)]\cap P_{P(u)}\right)$ and $i\in [k]$ be such that $M_i(v) \vartriangleleft v$. There are two
  cases to be considered.

  \begin{enumerate}
    \item $y \vartriangleleft M_i(y)$: by Proposition \ref{prop assorbimento}, $P(M_i(y))=PP^{M_i}(M_i(y))=P(y)=P(u)$, and, by the maximality of $y$,
  $M_i(y) \nless P(v)$, so the result follows by Lemma \ref{lemma complemento}.
    \item $y \vartriangleright M_i(y)$: if $P^{M_i}(v) \neq P(v)$, then $P^{M_i}(v) \not \in \imm(P)$  and the result is true by our inductive hypothesis, since $u \leqslant y<P(v) < P^{M_i}(v) \vartriangleleft v$.
    Let $P^{M_i}(v) =P(v)$. Since $K$ is Eulerian, there exists $w \vartriangleleft v$ such that $y \leqslant w$ and $w\neq P(v)$.
    Therefore $w\not \in \imm(P)$, otherwise $w=P(w) \leqslant P(v)$, and the result is proved.
    \end{enumerate}
\end{proof}

\begin{ex} \label{esempio involuzioni} In this example we refer to Hultman's articles \cite{hultman3} and \cite{hultman4} for results and references about twisted involutions.
  Let $\mathfrak{I}(\theta)$ be the set of twisted involutions of a finite Coxeter system $(W,S)$. Ordered by inducing the Bruhat order, this is an Eulerian poset.
Let $s\in S$. By \cite[Theorem 4.5]{hultman4}, the function $M_s : \mathfrak{I}(\theta) \rightarrow \mathfrak{I}(\theta)$ defined by
$$M_s(v)=\left\{
         \begin{array}{ll}
           vs, & \hbox{if $\theta(s)vs=v$;} \\
           \theta(s)vs, & \hbox{otherwise,}
         \end{array}
       \right.
$$ for all $v\in \mathfrak{I}(\theta)$, is a special matching of $\mathfrak{I}(\theta)$. Then $P^{M_s}$ is a projection and, by Theorem \ref{immagine graduata}, $\imm(P^{M_s})$
is graded with same rank function of $\mathfrak{I}(\theta)$. The  M\"obius function of this poset is given in Proposition \ref{moebius proiezioni}. Moreover, by Theorem \ref{complemento graduato}, $[\mathfrak{I}(\theta) \setminus \imm(P^{M_s})]\cup \{e\}$
is graded with same rank function of $\mathfrak{I}(\theta)$, where $e$ is the identity in $W$.

Let us consider $(W,S)=(S_4,\{s_1,s_2,s_3\})$, a Coxeter system of type $A_3$, and  $\theta=\id_W$. Then $\mathfrak{I}(\theta)=\mathrm{Invol}(S_4)$ is the poset of involutions of $S_4$; its Hasse
diagram is plotted in \cite[Figure 2.14]{BB}. Using SageMath we have found that $|\sm(\mathrm{Invol}(S_4))|=6$; three of them are $M_{s_1}$, $M_{s_2}$ and $M_{s_3}$.
The other ones are $M_1,M_2,M_3$ defined by:
\begin{enumerate}
  \item $M_1(e)=2134$, $M_1(4231)=4321$, $M_1(1432)=3412$, $M_1(1243)=2143$ and $M_1(1324)=3214$;
  \item $M_2(e)=1324$, $M_2(4231)=3412$, $M_2(1432)=1243$, $M_2(2134)=3214$ and $M_2(2143)=3412$;
 \item $M_3(e)=1243$, $M_3(4231)=4321$, $M_3(1432)=1324$, $M_3(2134)=2143$ and $M_3(3214)=3412$.
\end{enumerate}
We find also that $|M^{\mathrm{Invol}(S_4)}|=46$ and $|E(M^{\mathrm{Invol}(S_4)})|=22$.
\end{ex}

\subsection{Special projections in finite Coxeter groups} \label{ultima sezione}

In this section we give the motivating examples of our investigation. We refer to \cite{BB} and \cite{humphreysCoxeter} for notation and terminology
concerning Coxeter groups. We consider only finite Coxeter groups, although many of the results mentioned are true in the general case.

Let $(W,S)$ be a finite Coxeter system. This consists of a finite group $W$ with a presentation given by a set $S$ of involutive generators. For $w \in W$, the natural number $\ell(w)$ is the \emph{length} of $w$, relative to the presentation $(W,S)$.  For any $I,J\subseteq S$ we let
\begin{eqnarray*} W^J&:=&\{w\in W:\ell(ws)>\ell(w)~\forall~s\in J\},
\\ {^JW}&:=&\{w\in W:\ell(sw)>\ell(w)~\forall~s\in J\},
\\ {^IW^J} &:=& {^IW}\cap W^J.
\end{eqnarray*}



The subgroup of $W$ generated by $J\subseteq S$ is denoted with $W_J$. Such a subgroup is usually called a \emph{parabolic subgroup}. In particular, $W_S=W$ and $W_\varnothing =
\{e\}$, where $e$ is the identity in $W$. As sets, one can observe that $W^J \simeq W/W_J$ and ${^J}W \simeq W_J \backslash W$. For this reason, sometimes in the literature they are called \emph{right} and \emph{left} \emph{quotients}, while a set ${^IW^J}$ is called \emph{double quotient}, since ${^IW^J}$ is a set of representative of $W_I \backslash W / W_J$.

Given a finite Coxeter system $(W,S)$, we write $\leqslant$ for the \emph{Bruhat order} on the group $W$ (see \cite[Chapter~2]{BB} or \cite[Chapter~5]{humphreysCoxeter}).
This order can be defined by the subword property. The poset $(W,\leqslant)$ is graded, with rank function $\ell$, and it is Eulerian.
The set $W^J$, ordered by inducing the Bruhat order, has a unique maximal element
$w_0^J$.
For any $J\subseteq S$, each element $w\in W$ factorizes uniquely as
$w=w^Jw_J$, with $\ell(w)=\ell(w^J)+\ell(w_J)$, $w^J\in W^J$ and $w_J\in W_J$.
The maximal element of $W$, denoted by $w_0$, factorizes as $w_0=w_0^Jw_0(J)$, where $w_0(J)$ is the maximal element of the subgroup $W_J$.
The function $P^J:W \rightarrow W^J$ defined by the assignment $w \mapsto w^J$, for all $w\in W$, is a poset morphism (\cite[Proposition~2.5.1]{BB}). Notice that the factorization $w=w^Jw_J$ and the subword property imply that $P^J(w)\leqslant w$. Then $P^J$ is regressive and order preserving. Analogous results hold for the idempotent function $Q^J:W \rightarrow {^JW}$, defined similarly to $P^J$.
We have that $\imm(P^J)=W^J$, $\imm(Q^J)={^JW}$ and $P^JQ^I=Q^IP^J$, for all $I,J \subseteq S$ (see, e.g. \cite[Lemma~2.6]{sentinelli-isomoprhism}).
Moreover $\imm(Q^IP^J)={^IW^J}$ and
\begin{enumerate}
  \item $P^J_u=uW_J=[u,uw_0(J)]$,
  \item $Q^J_v=W_Jv=[v,w_0(J)v]$,
  \item $(Q^IP^J)_w=W_IwW_J=[w,(w_0(I))^{I\cap J_w}ww_0(J)]$,
\end{enumerate}
where $J_w:=\{wsw^{-1}:s\in J\}$, for all $u\in W^J$, $v\in {^JW}$, $w\in {^IW^J}$.

For $w=s_1\cdots s_k \in W$ it is well defined a function $P^w \in \Or(W)$, obtained  by setting (see, e.g. \cite[Section 4]{Sentinelli-Artin})
$$P^w:=P^{\{s_1\}}\cdots P^{\{s_k\}}.$$  The function $P^w$ is idempotent if and only if $w=w_0(J)$, for some $J\subseteq S$ (\cite[Lemma 4.8]{Sentinelli-Artin}); moreover
the equality $P^{w_0(J)}=P^J$ holds. Analogous results are true for the functions $Q^J$.
The \emph{$0$-Hecke monoid} $(W,\ast)$ (or \emph{Coxeter monoid}, see e.g. \cite{KenneyCoxetermonoid}) is then isomorphic to the monoids $W^R:=\{P^w: w\in W\}$ and $W^L:=\{Q^w: w\in W\}$, and  $E(W^R)=\{P^J: J \subseteq S\}$, $E(W^L)=\{Q^J: J \subseteq S\}$. We let $W^{LR} \subseteq \Or(W)$ to be the submonoid generated by the set of idempotents $\{Q^IP^J: I, J \subseteq S\}$.
\begin{lem} \label{lemma 0-Heche bilatero}
  Let $(W,S)$ be a finite Coxeter system.
  Then $E(W^{LR})=\{Q^IP^J: I, J \subseteq S\}$.
\end{lem}
\begin{proof}
  We already know that $\{Q^IP^J: I, J \subseteq S\} \subseteq E(W^{LR})$. Moreover $W^{LR}=\{Q^uP^v: u,v \in W\}$.
 Let $u,v \in W$ and $Q^uP^v$ be idempotent. Since the Coxeter monoid $(W,\ast)$ with the Bruhat order is an ordered monoid (see, e.g. \cite[Lemma~2]{KenneyCoxetermonoid}),
we have that $w \leqslant w \ast w$, for all $w\in W$.
Then the finiteness of $W$ implies the existence of $h,k \in \N$ such that $(Q^u)^h=(Q^u)^{h+1}$ and $(P^v)^k=(P^v)^{k+1}$; then $(Q^u)^h$ and $(P^v)^k$ are idempotent, i.e. $(Q^u)^h=Q^I$ and $(P^v)^k=P^J$, for some $I,J \subseteq S$.
Therefore, if we let $r:=\max\{h,k\}$,  $Q^uP^v=(Q^uP^v)^r=(Q^u)^r(P^v)^r=Q^IP^J$.
\end{proof}

A special matching $M$ of an interval $[u,v]$ in $(W,\leqslant)$ is a \emph{right multiplication matching} if there exists $s\in S$ such that $M(z)=zs$, for all $z\in [u,v]$.
Analogously is defined a \emph{left multiplication matching}. In general, not all special matching are multiplication mathings (see e.g. \cite[Section 5.6]{BB}).
If $M\in \sm(W)$ is a multiplication matching then $P^M\in \{P^s, Q^s\}$, for some $s\in S$.
Therefore the functions $P^J=P^{w_0(J)}$ and $Q^J=Q^{w_0(J)}$ are special idempotents; more in general the function $P^JQ^I : W \rightarrow {^IW^J}$ is a special idempotent, for all $I,J \subseteq S$.

In the following proposition we give the idempotents of the monoid $M^{S_n}$, where $S_n$ is the symmetric group of order $n!$
with its standard Coxeter presentation.
\begin{prop} \label{prop monoide simmetrico}
  Let $n>1$ and $(S_n,S)$ be a Coxeter system of type $A_{n-1}$. Then $$E(M^{S_n})=\{Q^IP^J: I,J \subseteq S\}.$$
\end{prop}
\begin{proof}
By \cite[Corollary 3.6]{bcm3}, all special matchings of $S_n$ are multiplication matchings. This implies the equality $M^{S_n}=W^{LR}$. Hence, by Lemma \ref{lemma 0-Heche bilatero}, $E(M^{S_n})=\{Q^IP^J: I,J \subseteq S\}$.
\end{proof}

\begin{oss}
  The special matchings of a Bruhat interval $[e,w]$ in a symmetric group $S_n$ are classified in \cite[Theorem 5.1]{brenti-cohomology}.
  A characterization of special matchings of a lower Bruhat interval in any Coxeter group is given in \cite{caselli marietti2}.
It could be interesting to investigate on the monoids $M^{[e,w]}$ and their idempotents.
\end{oss}
\begin{oss}
The submonoid of $M^W$ generated by $\{P^s:s \in S\}$ is the Coxeter monoid (also known as $0$-Hecke monoid; see \cite{denton1}, \cite{Sentinelli-Artin} and references therein). When $W$ is the symmetric group $S_{n+1}$,  the non-commuting graph of the set $\{P^J:J \subseteq S\}$ is proved to be $n$-universal for forests in \cite{Sentinelli-grafi},
and conjectured to be $n$-universal. A fortiori, the same can be asserted for the non-commuting graph of the set of special idempotents $E(M^{S_{n+1}})$.
\end{oss}

By Proposition \ref{prop monoide simmetrico}, for a Coxeter systems of type $A_{n-1}$,
the partition of Corollary \ref{corollario partizione} is a left, or right or a double coset partition of the symmetric group $S_n$.
In Figure \eqref{fig-S3} we exhibit the Hasse diagram of the lattice $E(M^{S_3})$, where $S_3$ is the symmetric group of order $6$, generated by the simple transpositions $\{s,t\}$.

\begin{figure} \begin{center}\begin{tikzpicture}
\matrix (a) [matrix of math nodes, column sep=0.2cm, row sep=0.6cm]{
 & & P^{sts} & & \\
P^sQ^t & P^sQ^s  & & P^tQ^t & P^tQ^s   \\
P^s & Q^t & & Q^s & P^t \\
& & \id_{S_3} & & \\};

\foreach \i/\j in {4-3/3-1, 4-3/3-2,4-3/3-4, 4-3/3-5,3-1/2-1,3-1/2-2,3-2/2-1,3-2/2-4,3-4/2-2,3-4/2-5,3-5/2-4,3-5/2-5,1-3/2-1, 1-3/2-2,1-3/2-4, 1-3/2-5}
    \draw (a-\i) -- (a-\j);
\end{tikzpicture} \caption{Hasse diagram of $E(M^{S_3})$.} \label{fig-S3} \end{center} \end{figure}
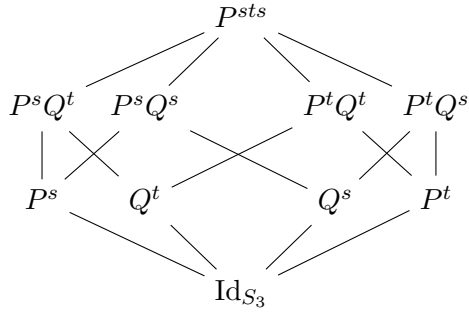

The following result, which is the content of \cite[Theorem 7.40]{OHara},  generalizes the notion of \emph{parabolic map}, introduced in \cite{billey}.
\begin{prop} \label{prop parabolic map}
  Let $(W,S)$ be a finite Coxeter system and $J \subseteq S$. Then the poset $uW_J \cap [u,v]$ has maximum, for all $u,v\in W^J$, $u \leqslant v$.
\end{prop}

\begin{cor} \label{teorema proiezioni}
  Let $(W,S)$ be a finite Coxeter system and $J \subseteq S$. Then $P^J$ and $Q^J$ are
special projections.
\end{cor}
\begin{proof}
  We have already observed that $P^J$ and $Q^J$ are special idempotents. Moreover $P^J_u=uW_J$ and $Q^J_v=W_Jv$,
for all $u\in W^J$, $v\in {^J}W$. Hence the result follows by Proposition \ref{prop parabolic map} and its left formulation.
\end{proof}

By Theorem \ref{immagine graduata} and Corollary \ref{teorema proiezioni} we deduce, in the finite case, the well known result that $W^J$ and ${^JW}$, with the induced Bruhat order, are graded. By Theorem \ref{complemento graduato} we also obtain that their complements are graded.
By Proposition \ref{moebius proiezioni}, we recover the formula for the M\"obius function of $W^J$ and ${^JW}$.
We also obtain the following result concerning double quotients.

\begin{prop}
  Let $(W,S)$ be a finite Coxeter system and  $I,J \subseteq S$ be such that $Q^IP^J$ is a projection. Then ${^IW^J}$, ordered by inducing the Bruhat order, is graded with
rank function $\ell$ and
$$\mu_{{^IW^J}}(u,v)=\left\{
                         \begin{array}{ll}
                           (-1)^{\ell(u,v)}, & \hbox{if $[u,v] \subseteq {^IW^J}$;} \\
                           0, & \hbox{otherwise,}
                         \end{array}
                       \right.
$$ for all $u,v\in {^IW^J}$, $u\leqslant v$.
\end{prop}

The following example shows that $\imm(P)$ could be graded, $P$ being a special idempotent but not a projection.
\begin{ex} \label{esempio non proiezione graduato}
  Let $S_3$ be the symmetric group of order $6$ generated by $\{s,t\}$. Then $P:=P^sQ^t$ is a special idempotent,
  $P_e=\{e,s,t,ts\}$ and $\imm(P)=\{e,st\}$. Then $P_e \cap [e,st]= \{e,s,t\}$;
  this set is not an interval. Hence $P$ is not a projection. Notice that $\imm(P)$ is graded, but the rank function is not $\ell$.
\end{ex}
In the next example we give all the special projections for the symmetric groups $S_2$, $S_3$, $S_4$, $S_5$ and $S_6$, with their standard Coxeter presentations.
The computations were performed with SageMath.
\begin{ex} \label{esempi proiezioni Sn}
  Let $1 \leqslant n \leqslant 5$. If $(S_{n+1},[n])$ is a Coxeter system of type $A_n$ and $I,J \in \mathcal{P}([n])\setminus \{\varnothing,[n]\}$,
  then $Q^IP^J$ is a special projection if and only if $I\cap J\in \{I,J\}$ and
$$(I,J) \in \left(\bigcup\limits_{k=1}^n  \left\{[k],[k,n]\right\}\right) \times \left(\bigcup\limits_{k=1}^n  \left\{[k],[k,n]\right\}\right),$$
where, if $a,b\in \N$ and $a \leqslant b$, we define $[a,b]:=\{z \in \N: a \leqslant z \leqslant b\}$.
\end{ex}

\subsection{Involutions and special projections} \label{sezione involuzioni}

Let $(W,S)$ be a finite Coxeter system and $\mathrm{Invol}(W):=\{w\in W: w=w^{-1}\}$. The set
$\mathrm{Invol}(W)$ with the induced Bruhat order is a graded poset; its rank function is
$r:=(\ell+a\ell)/2$, where $a\ell$ is the absolute length function  (see \cite[Exercise 2.35]{BB}).
We denote by $[u,v]^{\mathcal{I}}$ any Bruhat interval of $\mathrm{Invol}(W)$.
\begin{prop}
Let $J\subseteq S$ and $\mathrm{Invol}^J(W):=\mathrm{Invol}(W) \cap {^J}W^J$. Then
\begin{enumerate}
  \item $Q^JP^J : W \rightarrow {^J}W^J$ restricts to $Q^JP^J : \mathrm{Invol}(W) \rightarrow \mathrm{Invol}^J(W)$ and
  \item $\hat{P}^J:=Q^JP^J : \mathrm{Invol}(W) \rightarrow \mathrm{Invol}^J(W)$ is a special idempotent.
\end{enumerate}
\end{prop}
\begin{proof} We observe that, if $P^{M_s}$ is the idempotent defined in Example \ref{esempio involuzioni}, then $P^{M_s}=P^sQ^s$,
whenever $\theta=\id_W$.
Let $u \in \mathrm{Invol}(W)$ and $s_1 \cdots s_k=w_0(J)$ be a reduced word. Therefore
\begin{eqnarray*}
  Q^JP^J(u) &=& Q^{w_0(J)}P^{w_0(J)}(u)=Q^{s_1}P^{s_1} \cdots Q^{s_k}P^{s_k}(u) \\
 &=& P^{M_{s_1}}\cdots P^{M_{s_k}}(u) \in \mathrm{Invol}^J(W).\end{eqnarray*}\end{proof}
In the next proposition we prove that special projections in $\Or(W)$ of kind $Q^JP^J$ restrict to special projections in $\Or(\mathrm{Invol}(W))$.
\begin{prop} \label{prop proiezioni invol}
Let $(W,S)$ be a Coxeter system and $J \subseteq S$.
If $Q^JP^J \in M^W$ is a projection, then $\hat{P}^J \in M^{\mathrm{Invol}(W)}$ is a projection.
\end{prop}
\begin{proof} Recall that $u \leqslant v$ if and only if $u^{-1} \leqslant v^{-1}$, for all $u,v\in W$.
Let $Q^JP^J \in M^W$ be a projection and $z_{x,y}:=\max([x,y]\cap W_JxW_J)$, for all $x,y\in {^J}W^J$,
$x \leqslant y$.
Let $u,v \in \mathrm{Invol}^J(W)$ be such that $u \leqslant v$.
Then $u \leqslant z_{u,v}^{-1} \leqslant v$ and $Q^JP^J(z_{u,v}^{-1})=u^{-1}=u$;
hence $z_{u,v}^{-1} \in \max([u,v]\cap W_JuW_J)$, i.e. $z_{u,v}^{-1}=z_{u,v}$. This implies that
 $z_{u,v}$ is the maximum of $[u,v]^{\mathcal{I}}\cap \hat{P}^J_u$.
\end{proof}

The parabolic map $m: W \times \mathcal{P}(S) \rightarrow W$ is defined in \cite{billey}
by $$m(w,J)=\max\{v\in W_J: v \leqslant w\},$$ for all $(w,J) \in W \times \mathcal{P}(S)$. This function
restricts
to a parabolic map $$\hat{m}:\mathrm{Invol}(W)\times \mathcal{P}(S) \rightarrow \mathrm{Invol}(W).$$
as one can easily prove (see the proof of Proposition \ref{prop proiezioni invol}).
Nevertheless, an analogous of Proposition \ref{prop parabolic map} does not holds for involutions; the special projections of kind $\hat{P}^J$, for $\mathrm{Invol}(S_n)$, $n<8$, are given in the example below. The computations were performed with SageMath.
\begin{ex} \label{esempi proiezioni involuzioni Sn}
  Let $1 \leqslant n \leqslant 6$. If $(S_{n+1},[n])$ is a Coxeter system of type $A_n$ and $J \in \mathcal{P}([n])\setminus \{\varnothing,[n]\}$,
  then $\hat{P}^J$ is a projection if and only if
$$J \in \bigcup\limits_{k=1}^n  \left\{\{k\},[k],[k,n]\right\}.$$
\end{ex}

\section{Special partial idempotents} \label{partial}

Some of the results of Section \ref{sezione special} can be proved for the idempotents arising from \emph{special partial matchings}, a weaker notion of special matching introduced in \cite{Hultman1}. We recall the definition.
A function $M_p: K \rightarrow K$ is a \emph{partial matching}\footnote{In graph theory this is just a matching of the Hasse diagram satisfying the condition $M_p(\hat{1}) \vartriangleleft \hat{1}$.} if
\begin{enumerate}
  \item $M_p \circ M_p = \id_K$;
  \item $M_p(\hat{1}) \vartriangleleft \hat{1}$;
  \item $M_p(x) \vartriangleleft x$ or $x \vartriangleleft M_p(x)$ or $M_p(x)=x$, for all $x\in K$.
\end{enumerate}
A partial matching $M_p: K \rightarrow K$ is a \emph{partial special matching} if $x \vartriangleleft y$ and $x \neq M_p(y)$ implies $M_p(x) \leqslant M_p(y)$, for all $x,y\in K$.
We refer to \cite{Hultman2} and \cite{marietti-pirconi} for recent developments.
We let $\mathrm{m}_p(K)$ be the set of partial matchings of $K$ and $\sm_p(K)$ the one of partial special matchings. As for special matching,
a special partial matching satisfies a lifting property (see \cite[Lemma 5.2]{Hultman2}). Also the converse is true; as one can readily check,
a partial matching $M$ which satisfies the lifting property and the condition $M(\hat{1})\vartriangleleft \hat{1}$, is a special partial matching. We resume our observation in the following proposition.

\begin{prop}
  Let $M_p \in \mathrm{m}_p(K)$. Then $M_p \in \sm_p(K)$ if and only if $M_p$ satisfies the lifting property, i.e. $x\leqslant y$, $x \leqslant M_p(x)$ and $M_p(y) \leqslant y$ implies
$M_p(x) \leqslant y$ and $x \leqslant M_p(y)$, for all $x,y \in K$.
\end{prop}

If $M_p$ is a special partial matching of $K$, then we can define a \emph{special partial idempotent} $P^{M_p}: K \rightarrow K$ by letting $$P^{M_p}(x)= \left\{
          \begin{array}{ll}
            x, & \hbox{if $x \leqslant M_p(x)$;} \\
            M_p(x), & \hbox{if $M_p(x)\vartriangleleft x$,}
          \end{array}
        \right.
$$ for all $x\in K$. Then $P^{M_p}$ is regressive and order preserving, as one can see as in Proposition \ref{morfismo}.
Therefore, if we denote by $M_p^K$ the monoid generated by the set of idempotents $\{\id_K\}\cup\{P^M:M \in \sm_p(K)\}$, we have
$$M^K \subseteq M_p^K \subseteq \Or(K).$$

\begin{oss}
Let $M_p \in \sm_p(K)$. Since $P^{M_p} : K \rightarrow K$ is a poset morphism with small fibers (see \cite[Chapter 11]{kozlov}),
then $M_p$ is acyclic by \cite[Theorem 11.4]{kozlov}.
\end{oss}

\begin{ex} \label{esempio catena parziale}
  Let $n>1$ and $c_n$ be the chain of $n$ elements. For $n\geqslant 3$, the cardinality $|\sm_p(c_n)|$ is the number of matchings (Hosoya index) of the path graph $c_{n-2}$. It is not difficult to see that $|\sm_p(c_n)|=F_{n-1}$, for all $n\geqslant 2$, where $F_n$ is the $n$-th Fibonacci number.
As monoids, $M_p^{c_n} \simeq \{\id_{c_n}\}\cup \Or(c_{n-2}) $, for all $n \geqslant 3$. In fact, we can write $M_p^{c_n}=\{\id_{c_n}\}\cup\{fg:g\in \Or(c_{n-2})\}$, where $f: c_n \rightarrow c_n$ is the idempotent function  defined by setting $f(\hat{1}) \vartriangleleft \hat{1}$ and
$f(x)=x$ for all $x \leqslant f(\hat{1})$.
\end{ex}

\begin{thm} \label{teorema cartesiano parziale}
  Let $K_1,\ldots,K_n$ be finite graded posets and $K:= K_1 \times \ldots \times K_n$.
Then $M\in \sm_p(K)$ if and only if $$M=\id_{K_1}\times \ldots \times \id_{K_{i-1}} \times N \times  \id_{K_{i+1}} \times \ldots \times \id_{K_n},$$ for some $i\in [n]$, $N \in \sm_p(K_i)$.
\end{thm}
\begin{proof} It is sufficient to prove the result for $n=2$. Let $K=K_1\times K_2$ and $\hat{1}=(\hat{1}_1,\hat{1}_2)=\max K$.
We assert that $M(\hat{1}_1,\hat{1}_2)\in K_1 \times \{\hat{1}_2\}$ implies $M(K_1 \times \{\hat{1}_2\}) \subseteq K_1 \times \{\hat{1}_2\}$.
The proof is similar to the one for Theorem \ref{teorema cartesiano} and we omit it.

Let $M_1: K_1 \rightarrow K_1$ be the function defined by $M(x,\hat{1}_2)=(M_1(x),\hat{1}_2)$, for all $x\in K_1$; then $M_1\in \sm_p(K_1)$.
We prove that $M(x,y)=(M_1(x),y)$ for all $x,y$.
Let $(x,y)\in K$ be such that $y<\hat{1}_2$ and $M(u,v)= (M_1(u),v)$ for all $(u,v)>(x,y)$.
We consider only the case $M_1(x)=x$ and we prove that $M(x,y)=(x,y)=(M_1(x),y)$, the other cases being routinary.
It is easy to see that $M_1(x)=x$ implies $M(x,y) \leqslant (x,y)$. Let $M(x,y)<(x,y)$. We have $x<\hat{1}_1$, otherwise $M_1(x)<x$. Let $z \vartriangleright x$. If $M(x,y)=(x,u)$, $u \vartriangleleft y$, then $(x,y) \vartriangleleft M(z,u)=(M_1(z),u)$, a contradiction. Let $M(x,y)=(v,y)$, $v \vartriangleleft x$, and consider $w\in K_2$ such that $w \vartriangleright y$. Therefore $(v,y) \vartriangleleft (v,w)$ and $(x,y) \vartriangleleft M(v,w)$; hence $M(v,w)=(x,w)$, which implies $M_1(x)=v$, again a contradiction.
Therefore $M(x,y)=(x,y)$ and then $(M_1(x),y)=(x,y)=M(x,y)$.

Since either $M(\hat{1}_1,\hat{1}_2)\in K_1 \times \{\hat{1}_2\}$ or $M(\hat{1}_1,\hat{1}_2)\in \{\hat{1}_1\} \times K_2$, we have proved one implication in our statement.
The reverse implication is straightforward.
\end{proof}

\begin{cor} \label{corollario prodotto parziale}
   Let $K_1,\ldots,K_n$ be finite graded posets and $K:= K_1 \times \ldots \times K_n$.
Then, as monoids, $M_p^K \simeq M_p^{K_1}\times \ldots \times M_p^{K_n}$.
\end{cor}

\begin{ex}
 Let $n\in \N$, $n>1$ and $n=p_1^{k_1}\cdots p_h^{k_h}$ be the prime factorization of $n$. Let $P_n:=\{z\in \N: z\mid n\}$ ordered by divisibility.
Then, by Example \ref{esempio catena parziale} and Theorem \ref{teorema cartesiano parziale} we have that
$|\sm_p(P_n)|=\sum \limits_{i=1}^h F_{k_i}$. 
\end{ex}
The following proposition extends Proposition \ref{prop interv} and Corollary \ref{corollario partizione}. It says that a special partial idempotent is an interval retract,
as defined in Section \ref{sezione special}. The proof is analogous to the one of Proposition \ref{prop interv}.
\begin{prop} \label{prop interv parziale}
 Let $P \in \Or(K)$ be a special partial idempotent and $v\in \imm(P)$. Then $P_v=[v,v^P]$, for some $v^P\in P_v$. A function $\imm(P) \rightarrow K$ is defined by the assignment $v \mapsto v^P$ and
 $$K=\biguplus \limits_{v\in \imm(P)} [v,v^P].$$
\end{prop}


We end by extending Theorems \ref{immagine graduata} and \ref{complemento graduato}; we omit the proofs since they can be carried out in the same way. In fact, the arguments
use the definition of projection, the lifting property and the fact that $P_v$ is an interval, for all $v$ in the image of $P$.
\begin{thm} \label{immagine graduata parziale}
  Let $K$ be Eulerian and $P\in \Or(K)$ a special partial projection. Then the posets $\imm(P)$ and $[K\setminus \imm(P)]\cup \{\hat{0}\}$ are graded with rank function $\rho$.
\end{thm}

\end{document}